\documentclass[11pt,reqno]{amsart}

\usepackage{amsfonts,amsmath,amscd,amssymb,amsbsy,amsthm,amstext,amsopn,amsxtra,color,fullpage,mathrsfs,subfigure}
 \usepackage{verbatim} 
 \usepackage{stmaryrd}
 \usepackage{extarrows}
 \usepackage[all]{xy}
 \usepackage{bm}   
 \usepackage{hyperref}
 \usepackage{enumitem}
 \usepackage{mathtools}
\usepackage{cancel}
\usepackage{color}
\usepackage{tikz}
 \usepackage{tikz-cd}

 \numberwithin{equation}{section}
 \hypersetup{colorlinks,linkcolor=[rgb]{0,0,1},citecolor=[rgb]{1,0,0}}

\newtheorem{theorem}{Theorem}[section]
\newtheorem{lemma}[theorem]{Lemma}
\newtheorem{proposition}[theorem]{Proposition}
\newtheorem{corollary}[theorem]{Corollary}

\theoremstyle{definition}

\newtheorem{example}[theorem]{Example}

\newtheorem{remark}[theorem]{Remark}

\newcommand{\one}{\ensuremath{(\mathrm{i})}}
\newcommand{\two}{\ensuremath{(\mathrm{ii})}}
\newcommand{\three}{\ensuremath{(\mathrm{iii})}}
\newcommand{\four}{\ensuremath{(\mathrm{iv})}}

\newcommand{\CC}{\ensuremath{\mathbb{C}}} 
\newcommand{\kk}{\ensuremath{\Bbbk}} 
\newcommand{\NN}{\ensuremath{\mathbb{N}}} 
\newcommand{\QQ}{\ensuremath{\mathbb{Q}}}

\newcommand{\ZZ}{\ensuremath{\mathbb{Z}}} 
\newcommand{\M}{\ensuremath{\mathfrak{M}}}

\newcommand{\git}{\ensuremath{\operatorname{\!/\!\!/\!}}}
\newcommand{\GL}{\operatorname{GL}}
\newcommand{\head}{\operatorname{h}}
\newcommand{\Hilb}{\operatorname{Hilb}}
\newcommand{\hn}{\operatorname{\mathfrak{h}}} 

\newcommand{\Hom}{\operatorname{Hom}}

\newcommand{\Mov}{\operatorname{Mov}}
\newcommand{\Nef}{\operatorname{Nef}}

\newcommand{\Pic}{\operatorname{Pic}}
\newcommand{\Class}{\operatorname{Cl}}
\newcommand{\Proj}{\operatorname{Proj}}

\newcommand{\rank}{\operatorname{rk}}
\newcommand{\relint}{\operatorname{relint}} 

\newcommand{\Rep}{\operatorname{Rep}}
\newcommand{\Seshadri}{\operatorname{S}}
\newcommand{\SL}{\operatorname{SL}}

\newcommand{\Sym}{\operatorname{Sym}}
\newcommand{\supp}{\operatorname{supp}}
\newcommand{\tail}{\operatorname{t}}

\newcommand{\reg}{\operatorname{reg}}

\newcommand{\cone}[1]{\sigma_{#1}}
\newcommand{\ali}[1]{\textcolor{blue}{#1}}

\title{Hilbert schemes for crepant partial resolutions}

\author{Alastair Craw} 
\author{Ruth Wye} 

\address{Department of Mathematical Sciences, University of Bath, Claverton Down, Bath BA2 7AY, UK.}
\email{a.craw@bath.ac.uk; rep46@bath.ac.uk}
 
\begin{document}

\begin{abstract}
 For $n\geq 1$, we construct the underlying reduced subscheme of the Hilbert scheme of $n$ points on any crepant partial resolution of a Kleinian singularity as a Nakajima quiver variety for an explicit GIT stability parameter. This generalises and unifies existing quiver variety constructions of the Hilbert scheme of points on the minimal resolution of a Kleinian singularity, and on the Kleinian singularity itself. As a corollary, we compute the nef and movable cones of the Hilbert scheme 
 of $n$ points on any crepant partial resolution of a Kleinian singularity.
  \end{abstract}

 \maketitle
\tableofcontents

\section{Introduction}
Let $n\geq 1$ be a positive integer and let $\kk$ be an algebraically closed field of characteristic zero. A well-studied moduli space in algebraic geometry is the Hilbert scheme of $n$ points on a nonsingular surface $S$ over $\kk$, denoted $\Hilb^n(S)$. Fogarty~\cite{Fogarty68} showed that $\Hilb^n(S)$ is a nonsingular variety of dimension $2n$, while Beauville~\cite{Beauville83} later established that if $S$ is a symplectic surface, then $\Hilb^n(S)$ is a 
nonsingular symplectic variety. More generally, when $S$ is a normal surface with canonical singularities, the first author and Yamagishi~\cite{CY26} showed recently that the underlying reduced subscheme of $\Hilb^n(S)$ is a normal variety of dimension $2n$ with canonical singularities, and furthermore, if $S$ has symplectic singularities, then so does $\Hilb^n(S)$. Here, for the surfaces $S_K$ obtained as crepant partial resolutions of a Kleinian singularity, we establish a stronger statement by constructing each $\Hilb^n(S_K)$ as a Nakajima quiver variety. Throughout, we let $\Hilb^n(S_K)$ denote the underlying reduced subscheme of $\Hilb^n(S_K)$; in fact, the Hilbert scheme is reduced for $n\leq 7$ by \cite{CY26}.

To describe the quiver varieties of interest, let $\Gamma\subset \SL(2,\kk)$ be a nontrivial finite subgroup. Write the irreducible representations of $\Gamma$ as $\rho_0, \rho_1, \dots, \rho_r$ where $\rho_0$ is trivial, let $R(\Gamma)=\bigoplus_{0\leq i\leq r} \ZZ \rho_i$ be the representation ring, and write $\delta\in R(\Gamma)$ for the regular representation. Let $\Pi$ denote the preprojective algebra of a framed version of the affine ADE graph associated to $\Gamma$ by the McKay correspondence, and regard $v \coloneqq (1,n\delta)\in \ZZ\oplus R(\Gamma)$ as a dimension vector for $\Pi$-modules. Following Nakajima~\cite{Nakajima98}, the quiver variety $\M_\theta(\Pi,v)$ parametrises $\Pi$-modules of dimension vector $v$ that are semistable with respect to a stability parameter $\theta$ in $\Theta_v\cong \Hom(R(\Gamma),\QQ)$. The affine GIT quotient satisfies $\mathfrak{M}_0(\Pi,v)\cong \Sym^n(\mathbb{A}^2/\Gamma)$, and variation of GIT quotient gives a projective, crepant partial resolution
\[
\M_\theta(\Pi,v)\longrightarrow \Sym^n(\mathbb{A}^2/\Gamma).
\]
 For $n>1$, this morphism is a resolution if and only if $\theta$ lies in the interior of a top-dimensional cone of the GIT fan in $\Theta_v$; this wall-and-chamber decomposition is determined by a hyperplane arrangement that was computed explicitly in this case by Bellamy--Craw~\cite[Theorem~1.1]{BC20}.

To state our main result, let $S\to \mathbb{A}^2/\Gamma$ be the minimal resolution of the corresponding Kleinian singularity. Choose a subset $K\subseteq I\coloneqq\{1,\dots, r\}$ of the set of nodes in the ADE graph corresponding to the exceptional prime divisors in $S$, and write $S_K\to \mathbb{A}^2/\Gamma$ for the crepant partial resolution given by contracting the $(-2)$-curves in $S$ indexed by elements of $K$. Recall that the relative interior of a cone $\sigma$, denoted $\relint(\sigma)$, is the interior of $\sigma$ in its affine hull.
 
 \begin{theorem}
 \label{thm:mainnew}
 Let $n\geq 1$ and let $S_K\to \mathbb{A}^2/\Gamma$ be a crepant partial resolution. Then $\Hilb^n(S_K)$ is isomorphic to the quiver variety $\M_{\theta_K}(\Pi,v)$ for any $\theta_K$ in the relative interior of the GIT cone
 \[ 
 \sigma_K \coloneqq \big\{\theta\in \Theta_v \mid \theta(\delta) \geq 0, \theta(\rho_k)=0\text{ for }k\in K, \theta(\rho_i) \geq (n-1) \theta(\delta) \text{ for }i\in I\setminus K\big\}.
 \]
In particular, $\Hilb^n(S_K)$ is a normal, $\QQ$-factorial variety with symplectic singularities that admits a unique projective, symplectic (hence crepant) resolution obtained by variation of GIT quotient.  
 \end{theorem}

Two special cases of this result have appeared in the literature before. When $K$ is the emptyset, then $S_K$ is the minimal resolution $S$, the GIT cone $\sigma_K$ is of full dimension in $\Theta_v$, and the description of $\Hilb^n(S)$ as a nonsingular quiver variety was given by Kuznetsov~\cite[Theorem~43]{Kuznetsov07}, and also independently by both M.~Haiman and H.~Nakajima.  At the other extreme, when $K=I$, then $S_K$ is the Kleinian singularity $\mathbb{A}^2/\Gamma$, the GIT cone $\sigma_K$ is of dimension one, and the result was established by Craw--Yamagishi~\cite[Theorem~1.4]{CY26} (compare also \cite{CGS21}).

For each subset $K\subseteq I$, the GIT cone $\sigma_K$ from Theorem~\ref{thm:mainnew} lies in the simplicial cone
  \begin{equation}
     \label{eqn:Fintro}
      F:= \big\{ \theta\in \Theta_v \mid \theta(\delta)\geq 0, \;\theta(\rho_i)\geq 0 \text{ for }1\leq i\leq r\big\}
      \end{equation}
 with supporting hyperplanes $\delta^\perp\coloneqq \{\theta\in \Theta_v \mid \theta(\delta)=0\}$ and $\rho_i^\perp\coloneqq \{\theta\in \Theta_v \mid \theta(\rho_i)=0\}$ ($1\leq i\leq r$). When $n>1$, the main result from \cite[Theorem~1.2]{BC20} identifies the GIT chamber decomposition of $F$ with the Mori chamber decomposition of the movable cone of any projective, crepant resolution of $\Sym^n(\mathbb{A}^2/\Gamma)$. Theorem~\ref{thm:mainnew} shows that $\sigma_K$ is identified with the nef cone of $\Hilb^n(S_K)$.

\begin{example}
For $n=3$ and $\Gamma$ of type $A_2$, Figure~\ref{A2chamberdecompwithcones} illustrates the height-one slice of the GIT chamber decomposition of $F$.
\begin{figure} [!htb]
        \centering
 \begin{tikzpicture}[baseline={(0,-0.3)},xscale=1.25,yscale=1]
			\tikzset{>=latex}
    \draw (0,4) -- (4,4);
    \draw (2,0) -- (0,4);
     \draw (2,0) -- (4,4);
     \draw (4,4) -- (1,2);
     \draw (0,4) --(3,2);
     \draw (1,2) -- (3,2);
     \draw (4,4) -- (2/3,8/3);
     \draw (0,4) -- (10/3, 8/3);
     \draw (2/3, 8/3)-- (10/3,8/3);
   \node at (2,4.3) {$\delta^\perp$};
  \node at (0.6,2) {$\rho_1^\perp$};
 \node at (3.4,2) {$\rho_2^\perp$};
   \draw[color=red,line width = 0.75mm] (0,4) -- (2/3, 8/3);
   \draw[color=red,line width = 0.75mm] (4,4) -- (10/3, 8/3);
\node at (0,3.2){$\sigma_{\{1\}}$};
\node at (4.1, 3.2){$\sigma_{\{2\}}$};
\filldraw[red] (2,0) circle (2pt) node[black, anchor=north]{$\sigma_{\{1,2\}}$};
\filldraw [red] (0,4) -- (4,4) -- (2,3.21)-- (0,4);
\node at (2, 3.6){$\sigma_{\varnothing}$};
\end{tikzpicture}
        \caption{For $n=3$ and $\Gamma$ of type $A_2$: decomposition of $F\cong \Mov(\Hilb^3(S))$ into GIT cones, with the cones $\sigma_K$ defining $\Hilb^3(S_K)$ for $K\subseteq \{1,2\}$ coloured red.}
        \label{A2chamberdecompwithcones}
    \end{figure}
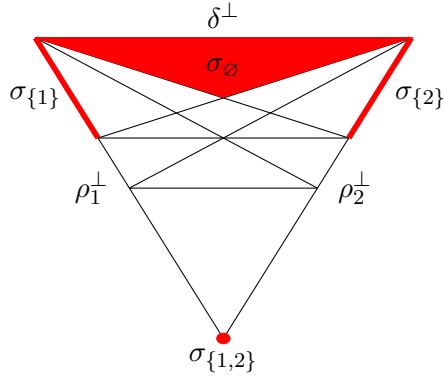 
 The top-dimensional cone $\sigma_{\varnothing}$ highlighted in Figure~\ref{A2chamberdecompwithcones} determines $\Hilb^3(S)$ as in \cite{Kuznetsov07} (compare \cite[Example~2.1(1)]{BC20}), the extremal ray $\sigma_{\{1,2\}}=\rho_1^\perp\cap \rho_2^\perp$ highlighted as a dot at the bottom of Figure~\ref{A2chamberdecompwithcones} determines $\Hilb^3(\mathbb{A}^2/\Gamma)$ as in \cite{CY26}, while the cones $\sigma_{\{k\}}$ for $k=1, 2$ highlighted in the top-left and top-right of Figure~\ref{A2chamberdecompwithcones} determine $\Hilb^3(S_{\{k\}})$, where $S_{\{k\}}$ is obtained from $S$ by contracting the $(-2)$-curve indexed by node $k$ of the Dynkin diagram of type $A_2$. 
  \end{example}

We now explain the proof of Theorem~\ref{thm:mainnew} for $n>1$. For each subset $K\subseteq I$, the intersection of $\sigma_K$ with the supporting hyperplane $\delta^\perp$ of $F$ is a GIT cone $\sigma_K^\prime\subseteq \sigma_K$ of codimension-one; for $K=I$, the cone $\sigma_I'$ is the origin in $\Theta_v$.
The catalyst for Theorem~\ref{thm:mainnew} is the observation that $\Sym^n(S_K)$ is isomorphic to the quiver variety $\mathfrak{M}_{\theta_K'}(\Pi,v)$ for any stability condition $\theta_K'$ in the relative interior of $\sigma_K'$ (see Proposition~\ref{prop:SymtoSym}). It follows that variation of GIT quotient (VGIT) defines a crepant partial resolution $\Sym^n(S_K)\to \Sym^n(\mathbb{A}^2/\Gamma)$. The composition of this morphism with the Hilbert--Chow morphism $\Hilb^n(S_K)\rightarrow \Sym^n(S_K)$ defines a crepant partial resolution 
   \[
\kappa\colon \Hilb^n(S_K)\longrightarrow \Sym^n(\mathbb{A}^2/\Gamma).
  \]
 We may now apply the main result of Bellamy--Craw~\cite[Theorem~1.2]{BC20} to obtain a GIT cone $\tau_K$ in $F$ such that $\Hilb^n(S_K)$ is isomorphic to the Nakajima quiver variety $\mathfrak{M}_{\theta}(\Pi,v)$ for any $\theta$ in the relative interior of $\tau_K$. Moreover, since $\kappa$ factors via the VGIT morphism from $\Sym^n(S_K)$ described above, it follows that the cone $\tau_K$ must contain $\sigma_K'$. In fact, analysing the class group of the Hilbert scheme of $n$ points on $S_K$ allows us to show that both $\sigma_K$ and $\tau_K$ are GIT cones in $F$ containing $\sigma_K'$ as a codimension-one face.
 The main statement of Theorem~\ref{thm:mainnew} follows once we show that 
 the cone $\sigma_K$ satisfies a collection of combinatorial properties that characterise  $\tau_K$ uniquely, giving $\sigma_K=\tau_K$. Our construction also produces a unique GIT chamber in $F$ containing $\sigma_K$ in its closure, from which it follows that $\Hilb^n(S_K)$ admits a unique projective crepant resolution obtained by variation of GIT quotient for quiver varieties. Forthcoming joint work of the second author with Bertsch~\cite{BW26}, which builds on the results we present here, provides a moduli space description of this projective crepant resolution phrased directly in terms of sheaves on the surface $S_K$.

 Theorem~\ref{thm:mainnew} allows us to describe explicitly the nef and movable cones of $\Hilb^n(S_K)$. To state the result, consider the GIT chamber $C_K$ (see Lemma~\ref{lem:C_K'}) where for $\theta\in C_K$, the VGIT morphism 
 \begin{equation}
     \label{eqn:crepantres}
 h\colon \mathfrak{M}_\theta\longrightarrow  \mathfrak{M}_{\theta_K}\cong \Hilb^n(S_K)
 \end{equation}
 is the unique projective crepant resolution. Write $\mathcal{R}_i$ for $i\in Q_0$ to denote the tautological bundles on $\mathfrak{M}_\theta$, and define the vector $\delta_J\coloneqq (\dim \rho_j)\in \mathbb{N}^J\coloneqq \bigoplus_{j\in J} \NN \rho_j$.

\begin{theorem}\label{thm:NefHilb}
 For $J\coloneqq \{0,1,\dots, r\}\setminus K$, pullback along the crepant resolution \eqref{eqn:crepantres} identifies:
 \begin{enumerate}
     \item[\one] the rational Picard group of $\Hilb^n(S_K)$ with the $\QQ$-span of $\{\det(\mathcal{R}_j) \mid j\in J\}$ in $\Pic(\mathfrak{M}_\theta)_\QQ;$ 
     \item[\two] the nef cone  of $\Hilb^n(S_K)$ with the cone 
 \[
   \Big\{\textstyle{\bigotimes_{j\in J} \det(\mathcal{R}_j)^{\otimes \eta_j}} \mid \eta(\rho_j) \geq (n-1) \eta(\delta_J)\geq 0 \text{ for }j\in J\setminus \{0\}\Big\};
    \]  
    \item[\three] the movable cone of $\Hilb^n(S_K)$ with the decomposition of the simplicial cone 
   \[
         \Big\{\textstyle{\bigotimes_{j\in J} \det(\mathcal{R}_j)^{\otimes \eta_j}} \mid 
     \eta(\rho_j) \geq 0 
     \text{ for }j\neq 0, \; 
     \eta(\delta_J)\geq 0\Big\}
    \]
    determined by the hyperplanes $\{(m\delta_{J}- \alpha)^\perp \mid \alpha\in \Phi_{I\setminus K}^+, 0\leq m < n\}$, where $\Phi_{I\setminus K}^+$ is the set of positive roots of the root system of type ADE associated to $\Gamma$ that are supported in $I\setminus K$.
 \end{enumerate} 
\end{theorem}

\medskip

\noindent \textbf{Acknowledgements:} The first author was supported by Research Project Grant RPG-2021-149 from The Leverhulme Trust. The second author was supported fully by a PhD studentship awarded by the Heilbronn Institute for Mathematical Research, Grant number EP/V521917/1, from UKRI.

\section{On Nakajima quiver varieties for Kleinian singularities}
\label{section:NQV}
In this section, we provide some background on Nakajima quiver varieties, and we generalise a couple of results from work of the first author with Bellamy~\cite{BC20} and with Bellamy--Schedler~\cite{BCS26}.

\subsection{Nakajima quiver varieties}
Let $\Gamma\subset \SL(2,\kk)$ be a nontrivial finite subgroup. List the irreducible representations of $\Gamma$ as $\rho_0, \rho_1, \dots, \rho_r$ where $\rho_0$ is trivial. The McKay graph of $\Gamma$ is defined to have nodes set $\{0,1,\dots, r\}$, and $\dim_\Gamma \Hom(\rho_i,\rho_j\otimes V)$ edges joining vertex $i$ to vertex $j$, where $V$ is the given $2$-dimensional representation of $\Gamma$. Consider the symmetric matrix $C_\Gamma:= 2\operatorname{Id}-A_\Gamma$, where $A_\Gamma$ is the adjacency matrix of this graph. McKay~\cite{McKay80} showed that by equipping $R(\Gamma):=\bigoplus_{0\leq i\leq r} \ZZ \rho_i$ with the symmetric bilinear form determined by $C_\Gamma$, we obtain the root lattice of an affine root system of type ADE, where the Dynkin diagram is the McKay graph, the simple roots are the irreducible representations of $\Gamma$, and the minimal imaginary root $\delta=(\delta_i)\in R(\Gamma)$ satisfies $\delta_i=\dim \rho_i$ for $0\leq i\leq r$. The integer span of the nontrivial irreducible representations defines $\Phi$, the corresponding root system of finite type. Write $\Phi^+$ for the set of positive roots.

For the main quiver of interest, augment the McKay graph by adding a framing vertex $\infty$ and an edge joining $\infty$ to $0$.  The \emph{framed McKay quiver} $Q$ is defined to have vertex set $Q_0=\{\infty,0,1,\dots, r\}$, and arrows set $Q_1$ given by the set of pairs comprising an edge in the framed McKay graph and an orientation of the edge. Thus, $Q$ is obtained from the framed McKay graph by replacing each edge by two arrows, one in each direction. Write $\head, \tail\colon Q_1\to Q_0$ for the maps assigning to each arrow the head and tail vertex respectively of the corresponding oriented edge. For each $a\in Q_1$, write $a^*$ for the arrow corresponding to the same edge with opposite orientation. Let $\epsilon\colon Q_1\to \{\pm 1\}$ be any map satisfying $\epsilon(a)\neq \epsilon(a^*)$ for each $a\in Q_1$. The preprojective algebra $\Pi$ is the noncommutative algebra obtained as the quotient of the path algebra $\kk Q$ by the two-sided ideal of relations 
\[
 \left(\sum_{\{a\in Q_1 \mid \head(a)=i\}} \epsilon(a)aa^* \mid i\in Q_0\right).
 \]
The \emph{McKay quiver} $Q_\Gamma$ is the complete subquiver of $Q$ on the vertex set $\{0,1,\dots, r\}$.  The preprojective algebra $\Pi_\Gamma$ for $Q_\Gamma$ is isomorphic to $\Pi/(e_\infty)$, where $e_\infty$ is the vertex idempotent corresponding to the trivial path at $\infty$. 

 Let $\rho_\infty$ be a formal symbol and write $\ZZ^{Q_0}=\ZZ\rho_\infty\oplus R(\Gamma)$ for the space of dimension vectors. We write elements of $\ZZ^{Q_0}$ in the form $\alpha=\sum_{i\in Q_0} \alpha_i\rho_i$, where $\alpha_i\in \ZZ$ for $i\in Q_0$, or alternatively, in the form $\alpha=(\alpha_\infty, \sum_{0\leq i\leq r} \alpha_i\rho_i)$ to emphasise the component at the framing vertex. Fix a natural number $n \geq 1$, and define the vector
\[
  v:=(1,n\delta)=\rho_\infty+n\sum_{0\leq i\leq r}\dim(\rho_i)\rho_i\in\ZZ^{Q_0}.
\]
 
 To construct the quiver varieties of interest, let $\Rep(Q,v):=\bigoplus_{a\in Q_1} \Hom(\kk^{v_{\tail(a)}},\kk^{v_{\head(a)}})$ denote the vector space of representations of the framed McKay quiver that have dimension vector $v$. Write $\Rep(\Pi,v)$ for the closed subscheme of $\Rep(Q,v)$ parametrising those representations that satisfy the preprojective relations, or equivalently, $\Pi$-modules of dimension vector $v$. The reductive group $G(v):= \GL(1)\times \prod_{0\leq i\leq r} \GL(n\dim\rho_i)$ acts on $\Rep(Q,v)$ by conjugation, leaving the affine subscheme $\Rep(\Pi,v)$ invariant. The one-dimensional diagonal scalar subgroup $\kk^\times\subset G(v)$ acts trivially. It follows that, if we define the rational vector space 
 \[
 \Theta_v:= \big\{\theta\in \Hom_\ZZ(\ZZ^{Q_0},\QQ) \mid \theta(v)=0\big\},
 \]
  then every character of $G(v)/\kk^\times$ is of the form $\chi_\theta\colon G(v)/\kk^\times\to \CC^\times$ for some integer-valued $\theta\in \Theta_v$, where $\chi_\theta(g) = \prod_{i \in Q_0} \det(g_i)^{\theta_i}$ for $g=(g_i)\in G(v)/\kk^\times$. For $\theta\in \Theta_v$, after replacing $\theta$ by a positive multiple if necessary, we consider the categorical quotient 
  \[
  \mathfrak{M}_\theta(\Pi,v):= \Rep(\Pi,v)\git_{\chi_\theta} G(v) = \Proj \bigoplus_{k\geq 0} \kk[\Rep(\Pi,v)]^{k\theta},
 \]
 where $\kk[\Rep(\Pi,v)]^{\theta}$ is the vector space of $\chi_{\theta}$-semi-invariant functions. The scheme $\mathfrak{M}_\theta(\Pi,v)$ is reduced and irreducible by Gan--Ginzburg~\cite[Theorem~1.6]{GanGinzburg06}  and Bellamy--Schedler~\cite[Theorem~1.2]{BS21} respectively, so we refer to $\mathfrak{M}_\theta(\Pi,v)$ as the \emph{Nakajima quiver variety} associated to $\theta$. By applying the isomorphism from \cite[Lemma~4.5]{BC20}, we see that variation of GIT quotient (VGIT) induces a projective morphism from $\mathfrak{M}_\theta(\Pi,v)$ to the affine quotient 
  \[
  \mathfrak{M}_0(\Pi,v) \cong \Sym^n(\mathbb{A}^2/\Gamma).
  \]

   For $\theta\in \Theta_v$, a $\Pi$-module $B$ of dimension vector $v$ is \emph{$\theta$-semistable} (resp.\ \emph{$\theta$-stable}) if, for any non-zero proper $\Pi$-submodule $B'\subset B$, we have $\theta(B')\geq 0$ (resp.\ $\theta(B')>0$). We say that the stability condition $\theta$ is \emph{generic} if every $\theta$-semistable $\Pi$ module of dimension $v$ is $\theta$-stable. For a fixed $\theta \in \Theta_v$, two $\theta$-semistable $\Pi$-modules $B, B'$ are said to be \emph{$\Seshadri$-equivalent} if they admit filtrations
  \[
  0=B_0 \subset B_1\subset \cdots B_\ell=B\quad \text{and}\quad 0=B'_0 \subset B'_1\subset \cdots B'_{\ell'}=B'
  \]
  such that each $B_i$ and $B_i'$ is $\theta$-semistable, and $\bigoplus_{1\leq i\leq \ell} B_i/B_{i-1}\cong \bigoplus_{1\leq i\leq \ell'} B'_i/B'_{i-1}$. The \emph{$\theta$-polystable module} in each $\Seshadri$-equivalence class is the unique (up to isomorphism) representative that is a direct sum of $\theta$-stable modules. Following King~\cite{King94}, the GIT quotient $ \mathfrak{M}_\theta(\Pi,v)$ may be regarded as the coarse moduli space of $\theta$-semistable $\Pi$-modules of dimension vector $v$, up to $\Seshadri$-equivalence. 
 
 \subsection{Birational geometry}
 For $\gamma\in \ZZ\oplus R(\Gamma)$, write $\gamma^\perp=\{\theta\in \Theta_v \mid \theta(\gamma)=0\}$ for the dual hyperplane in $\Theta_v$. It was shown by Bellamy--Craw~\cite[Theorem~1.1]{BC20} that the hyperplane arrangement
  \begin{equation}
  \label{eqn:mathcalA}
  \mathcal{A}:=\big\{\delta^\perp, (m\delta\pm \alpha)^\perp \mid \alpha\in \Phi^+, 0\leq m < n\big\}
  \end{equation}
  determines the GIT fan in $\Theta_v$ for the action of $G(v)$ on $\Rep(\Pi,v)$, where parameters $\theta, \theta'\in \Theta_v$ lie in the relative interior of the same closed GIT cone $\sigma$ if and only if the notions of $\theta$-semistability and $\theta'$-semistability coincide, in which case $\mathfrak{M}_\theta(\Pi,v)\cong \mathfrak{M}_{\theta'}(\Pi,v)$. The \emph{GIT chambers} are the interiors of the top-dimensional cones in this fan, and $\theta$ is generic if and only if it lies in a GIT chamber, in which case $\M_\theta(\Pi,v)$ is nonsingular. It was shown in \cite[Proposition~2.2]{BC20}, that the simplicial cone
  \begin{equation}
     \label{eqn:F}
      F:= \big\{ \theta\in \Theta_v \mid \theta(\delta)\geq 0, \;\theta(\rho_i)\geq 0 \text{ for }1\leq i\leq r\big\}
      \end{equation}
  is a fundamental domain for the action of the Namikawa--Weyl group. We note that $F$ is a union of the closures of GIT chambers, and that the only hyperplanes in $\mathcal{A}$ that intersect the interior of $F$ are those of the form $(m\delta-\alpha)^\perp$ for $\alpha\in \Phi^+$ and $0\leq m< n$.

  Let $\theta\in \Theta_v$ be generic, and let $C$ be the GIT chamber containing $\theta$. Since $v$ is indivisible, \cite[Proposition~5.3]{King94} shows that $\mathfrak{M}_\theta(\Pi,v)$ is the fine moduli space of $\theta$-stable $\Pi$-modules of dimension $v$, up to isomorphism.  In the case, $\mathfrak{M}_\theta(\Pi,v)$ carries a tautological locally-free sheaf $\mathcal{R}=\bigoplus_{i\in Q_0} \mathcal{R}_i$, where $\mathcal{R}_\infty\cong \mathcal{O}_{\mathfrak{M}_\theta(\Pi,v)}$ and $\rank(\mathcal{R}_i)=n\dim\rho_i$ for $0\leq i\leq r$. The map of rational vector spaces
  \begin{equation}
  \label{eqn:linearisation}
  L_C\colon \Theta_v \longrightarrow \Pic\big(\mathfrak{M}_\theta(\Pi,v)\big)\otimes_{\ZZ}\QQ
 \end{equation}
  defined by sending any integer-valued $\eta\in \Theta_v$ to the line bundle $L_C(\eta) =\bigotimes_{0\leq i\leq r} \det(\mathcal{R}_i)^{\otimes \eta_i}$ is the \emph{linearisation map} associated to the chamber $C$. Note that $\Pic(\mathfrak{M}_\theta(\Pi,v))$ coincides with the relative Picard group of $\mathfrak{M}_\theta(\Pi,v)$ over $\mathfrak{M}_0(\Pi,v)$ because $\Pic(\Sym^n(\mathbb{A}^2/\Gamma))$ is trivial~\cite[Theorem~3.6.1]{Benson93}.
 
 We require the following slightly stronger form of the main result from Bellamy--Craw~\cite{BC20}.

  \begin{theorem}
  \label{thm:BC20}
For $n>1$, let $\theta\in F$ be generic, and let $C$ denote the GIT chamber containing $\theta$. For the quiver variety $X:=\mathfrak{M}_\theta(\Pi,v)$, we have that: 
 \begin{enumerate}
     \item[\one] the map $L_C$ from \eqref{eqn:linearisation} is an isomorphism that identifies the GIT wall-and-chamber structure in $F$ with the decomposition of the movable cone of $X$ 
     into Mori chambers;
     \item[\two] for every $\eta\in F$, the quiver variety $\mathfrak{M}_\eta(\Pi,v)$ is the birational model of $X$ determined by the line bundle $L_C(\eta)$; and
     \item[\three] $X$ is a Mori Dream Space over $\Sym^n(\mathbb{A}^2/\Gamma)$.
      \end{enumerate}
 In particular, every projective crepant partial resolution of $\Sym^n(\mathbb{A}^2/\Gamma)$ is isomorphic to a quiver variety $\mathfrak{M}_\eta(\Pi,v)$ for some $\eta\in F$.
     \end{theorem}
      
 \begin{proof}
 In light of \cite[Theorem~1.2, Corollary~6.5]{BC20}, it remains to show that part \two\ holds when $\eta$ is not generic\footnote{An alternative approach is to deduce that \cite[Condition~3.3]{BCS26} holds in our setting by applying various results from \cite[Sections~5-6]{BC20}, then applying \cite[Lemma~3.9(ii)]{BCS26} in combination with Zariski's Main theorem and the fact that each quiver variety $\mathfrak{M}_\theta(\Pi,v)$ is normal by  both \cite{GanGinzburg06} and \cite{BS21}. For convenience, we provide a more direct proof here.} Let $C'$ be a GIT chamber in $F$ with $\eta\in \overline{C'}$, and let $\theta'\in C'$. Then $L_{C'}(\eta)$ is semiample on $X':=\mathfrak{M}_{\theta'}(\Pi,v)$ by \cite[Corollary~6.5\one]{BC20}. If we write  $f^\prime \colon \mathfrak{M}_{\theta'}(\Pi,v)\to \Sym^n(\mathbb{A}^2/\Gamma)$ for the structure morphism, then the section ring $R(X', L_{C'}(\eta))=\bigoplus_{m\geq 0} (f^\prime)_*L_{C'}(\eta)^{\otimes m}$ is therefore finitely generated by a theorem of Zariski~\cite[Example 2.1.30]{LazarsfeldI}, so the birational model $\Proj R(X', L_{C'}(\eta))$ is well-defined.  Since $\eta$ lies in a face of $C'$, \cite[(6.3)]{BC20} shows that after replacing $\eta$ by some positive multiple if necessary, the induced morphism $g\colon \mathfrak{M}_{\theta'}(\Pi,v)\to \mathfrak{M}_{\eta}(\Pi,v)$ satisfies 
 \[
 L_{C'}(\eta) = g^*(\mathcal{O}_{\mathfrak{M}_{\eta}(\Pi,v)}(1)).
 \]
 The morphim $f_\eta\colon \mathfrak{M}_{\eta}(\Pi,v)\to \Sym^n(\mathbb{A}^2/\Gamma)$ obtained by variation of GIT quotient satisfies $f^\prime=f_\eta\circ g$. Since $g_*(\mathcal{O}_{\mathfrak{M}_{\theta'}(\Pi,v)})\cong \mathcal{O}_{\mathfrak{M}_\eta(\Pi,v)}$, we have for each $m\geq 0$ that
 \[
 (f^\prime)_*L_{C'}(\eta)^m\cong 
 (f_\eta)_*\big(g_*(g^*(\mathcal{O}_{\mathfrak{M}_{\eta}(\Pi,v)}(m))\big)
 \cong
 (f_\eta)_*(\mathcal{O}(m)\otimes g_*(\mathcal{O}_{\mathfrak{M}_{\theta'}(\Pi,v)})
 \cong 
  (f_\eta)_*\mathcal{O}(m)
 \]
 on $\Sym^n(\mathbb{A}^2/\Gamma)$, so the section ring $R(X^\prime, L_{C^\prime}(\eta))$ is isomorphic to the section ring of the polarising ample bundle $\mathcal{O}(1)$ on  $\mathfrak{M}_{\eta}(\Pi,v)$. It follows that
 $\Proj R\big(X',L_{C'}(\eta)\big)\cong \mathfrak{M}_{\eta}(\Pi,v)$. Since $C$ and $C'$ lie in $F$, compatibility of the linearisation maps $L_{C}$ and $L_{C'}$ from \cite[(6.5)]{BC20} gives $L_C(\eta)\cong \phi^*(L_{C'}(\eta))$, where $\phi\colon \mathfrak{M}_\theta(\Pi,v)\dashrightarrow \mathfrak{M}_{\theta'}(\Pi,v)$ is the birational map given by VGIT. Therefore, $\phi^*$ identifies the section rings of $L_C(\eta)$ and $L_{C'}(\eta)$, giving $\Proj R(X,L_{C}(\eta)) \cong \Proj R(X',L_{C'}(\eta))\cong \mathfrak{M}_{\eta}(\Pi,v)$.
 \end{proof}

 We now record two results that mirror and extend slightly those from \cite[Corollary~3.15\three-\four]{BCS26}.

 \begin{corollary}
 \label{cor:PicdimC}
 Let $n>1$. For $\eta\in F$, let $F'$ be the unique face of $F$ satisfying $\eta\in \relint(F')$, and let $\sigma$ be the unique GIT cone satisfying $\eta\in \relint(\sigma)$. Then 
    \begin{enumerate}
        \item[\one] the rank of $\Pic(\mathfrak{M}_\eta(\Pi,v))$ equals $\dim(\sigma)$; and    \item[\two] the variety $\mathfrak{M}_\eta(\Pi,v)$ is $\QQ$-factorial if and only if $\dim(\sigma)=\dim(F')$. 
    \end{enumerate} 
 \end{corollary}
 \begin{proof}
  Let $\sigma'\subset F$ be any GIT cone satisfying $\sigma\subseteq \sigma'\subseteq F'$ and $\dim(\sigma')=\dim(F')$. In addition, let $C\subset F$ be a GIT chamber with $\sigma'\subseteq \overline{C}$. Let $\eta'\in\relint(\sigma')$ and $\theta\in C$, and write
    \[
    \mathfrak{M}_{\theta}(\Pi,v)
    \stackrel{\omega}{\longrightarrow}
    \mathfrak{M}_{\eta'}(\Pi,v)
    \stackrel{\xi}{\longrightarrow}
    \mathfrak{M}_{\eta}(\Pi,v)
    \]
    for the VGIT morphisms, where $\mathfrak{M}_{\theta}(\Pi,v)$ is nonsingular because $\theta$ is generic. The map $L_{C}$ identifies $\sigma$ with the face $L_{C}(\sigma)$ of $\Nef(\mathfrak{M}_{\theta}(\Pi,v))$. The birational model of $X:=\mathfrak{M}_\theta(\Pi,v)$ determined by $L_{C}(\eta)$ is $\mathfrak{M}_{\eta}(\Pi,v)$, so 
    \begin{equation}
    \label{eqn:NefconeFace}
    L_{C}(\sigma)=(\xi\circ \omega)^*\big(\Nef(\mathfrak{M}_{\eta}(\Pi,v))\big)
    \end{equation}
    by Theorem~\ref{thm:BC20}\two. The map $(\xi\circ \omega)^*$ is injective by \cite[Corollary~4.2]{Ohta20}, so it identifies $\Nef(\mathfrak{M}_\eta(\Pi,v))$ with $L_{C}(\sigma)$. Part \one\ follows by equating the dimensions of these cones.  For part \two, apply the same logic to $\sigma'$ to see that $L_{C}(\sigma')=\omega^*(\Nef(\mathfrak{M}_{\eta'}(\Pi,v)))$. Moreover, since $\sigma'$ is a top-dimensional cone in the face $F'$ of $F$, applying the same logic to each GIT cone in the face $F'$ of $F$ gives
    \begin{equation}
        \label{eqn:pullbackMov}
    L_{C}(F')=\omega^*\big(\Mov(\mathfrak{M}_{\eta'}(\Pi,v))\big).
    \end{equation}
    The face $F'$ is the intersection of $F$ with $d:=(r+1)-\dim(F')$ supporting hyperplanes of $F$, and since $\dim(\sigma')=\dim(F')$, we see that $\omega$ is the composition of a sequence of $d$ divisorial contractions, so $\mathfrak{M}_{\eta'}(\Pi,v)$ is $\QQ$-factorial. If $\dim(\sigma)=\dim(F')$, then we may choose $\sigma':=\sigma$,  in which case $\mathfrak{M}_{\eta}(\Pi,v)$ is $\QQ$-factorial. Conversely, if $\dim(\sigma)<\dim(F')$, then $\xi^*(\Nef(\mathfrak{M}_{\eta}(\Pi,v)))$ is not of top-dimension in $\Mov(\mathfrak{M}_{\eta'}(\Pi,v))$, so $\xi$ is a small contraction and $\mathfrak{M}_{\eta'}(\Pi,v)$ is not $\QQ$-factorial.
    \end{proof}

 \subsection{The surface case} 
 \label{sec:n=1}
    When $n=1$, the analogue of Theorem~\ref{thm:BC20} holds verbatim except for part \one, where $L_C$ is merely a surjective map from the GIT fan to the Mori decomposition of the movable cone with a one-dimensional kernel spanned by $(-1,1,0,\dots, 0)\in \Theta_v$ (see \cite[Proposition~7.11]{BC20}). In that case, the map $\mathcal{R}_\infty\to \mathcal{R}_0$ of tautological bundles on $\mathfrak{M}_\theta(\Pi,v)$ for $\theta\in C$ is an isomorphism; to put it another way, adding the framing vertex has introduced some redundancy.

 For a cleaner statement when $n=1$, we replace $Q$, $\Pi$, $v$, $G(v)$ and $\theta\in \Theta_v$ by the McKay quiver $Q_\Gamma$, the preprojective algebra $\Pi_\Gamma$, the vector $\delta$, the reductive group $G(\delta):=\prod_{0\leq i\leq r} \GL(\dim(\rho_i))$, and a stability condition $\zeta\in \Theta_\delta$ respectively; the analogous construction defines the McKay quiver moduli spaces $\mathfrak{M}_\zeta(\Pi_\Gamma,\delta)$. For a GIT chamber $C\subset \Theta_\delta$ and for $\zeta\in C$, the tautological bundle on $\mathfrak{M}_\zeta(\Pi_\Gamma,\delta)$ has one fewer summands than before, and we write it as  $\mathcal{V}=\bigoplus_{0\leq i\leq r} \mathcal{V}_i$, where $\mathcal{V}_0\cong \mathcal{O}_{\mathfrak{M}_\zeta(\Pi_\Gamma,\delta)}$ and $\rank(\mathcal{V}_i)=\dim\rho_i$ for $1\leq i\leq r$. We adopt the same notation 
 \[
 L_C\colon \Theta_\delta \longrightarrow \Pic\big(\mathfrak{M}_\zeta(\Pi_\Gamma,\delta)\big)\otimes_\ZZ \QQ
 \]
 for the linearisation map, where each integer-valued $\eta\in \Theta_\delta$ satisfies $L_C(\eta) = \bigotimes_{1\leq i\leq r} \det(\mathcal{V}_i)^{\otimes \eta_i}$. The following analogue of Theorem~\ref{thm:BC20} is well-known: it was established originally over the field $\CC$ by Kronheimer~\cite{Kronheimer89} using hyperk\"{a}hler reduction.
  
\begin{theorem}[Kronheimer]
\label{thm:Kronheimer}
Let $\Gamma\subset \SL(2,\kk)$ be a nontrivial finite subgroup and let $f\colon S\to \mathbb{A}^2/\Gamma$ be the minimal resolution of the Kleinian singularity. Let $\zeta\in \Theta_\delta$ be a generic stability condition, and let $C$ denote the chamber in $\Theta_\delta$ containing $\zeta$. There is a commutative diagram
    \begin{equation}
        \label{eqn:Kronheimer}
    \begin{tikzcd}
    S\arrow[d]\arrow[r, "f"] & \mathbb{A}^2/\Gamma \arrow[d] \\
\mathfrak{M}_\zeta(\Pi_\Gamma,\delta)\arrow[r] & \mathfrak{M}_0(\Pi_\Gamma,\delta)
    \end{tikzcd}
    \end{equation}
    where the vertical maps are isomorphisms and the lower horizontal map is induced by VGIT. In addition, the GIT chamber decomposition of $\Theta_\delta$ coincides with the Weyl chamber decomposition of type ADE, the linearisation map $L_C$ identifies $\overline{C}$ with the nef cone of $S$, and for each $\eta\in \overline{C}$, the quiver variety $\mathfrak{M}_\eta(\Pi_\Gamma,\delta)$ is the birational model of $S$ determined by the line bundle $L_C(\eta)$.
\end{theorem}
\begin{proof}
  The dimension vector $\delta$ satisfies Crawley-Boevey's condition for $\Pi_\Gamma$, that is, $\delta\in \Sigma_0$, so a self-contained proof (of more general results) appears in \cite[Theorems~1.2, 4.6 and 4.18]{BCS26}.
\end{proof}
\begin{remark}
\label{rem:0semistable}
    The right-hand isomorphism in \eqref{eqn:Kronheimer} sends any point $z$ in $\mathbb{A}^2/\Gamma$ to the $\Seshadri$-equivalence class of the 0-semistable $\Pi$-module of dimension vector $\delta$ associated to any point in $f^{-1}(z)$.
\end{remark}

Theorem~\ref{thm:Kronheimer} implies that all crepant partial resolutions of $\mathbb{A}^2/\Gamma$ are Nakajima quiver varieties. To see this, we first introduce our notation for crepant partial resolutions. Recall that the exceptional locus of $f\colon S\to \mathbb{A}^2/\Gamma$ is a tree of $(-2)$-curves. The resolution graph,  whose nodes set $I=\{1,\dots, r\}$ is the set of such curves and an edge joins two nodes when the corresponding curves intersect, is a Dynkin diagram of type ADE.  Let $K\subseteq I=\{1,\dots, r\}$ be any subset.  Contracting only those $(-2)$-curves indexed by elements in $K$ determines a factorisation of the minimal resolution $f$ as
\begin{equation}
 \label{eqn:partialresolution}
\begin{tikzcd}
 S \arrow[r] & S_K\arrow[r,"g"] &  \mathbb{A}^2/\Gamma,
\end{tikzcd}
\end{equation}
 where $g$ is a crepant partial resolution. Every crepant partial resolution arises in this way, and every such surface $S_K$ has only Kleinian, hence symplectic, singularities \cite{Beauville00}. To obtain the quiver moduli space description of $S_K$, define a stability condition $\zeta_K\in \Theta_\delta$ for $\Pi_\Gamma$-modules by setting
 \begin{equation}
 \label{eqn:zetaK}
 \zeta_K(\rho_i) = \left\{ \begin{array}{cr} 0 & i\in K \\ 1 & i\in I\setminus K\end{array}\right.,
 \end{equation}
 so $\zeta_K(\rho_0) = -\sum_{i\in I\setminus K} \dim \rho_i$. 

\begin{corollary}
\label{cor:S_K}
For any subset $K\subseteq I=\{1,\dots, r\}$, there is a generic stability condition $\zeta\in \Theta_\delta$ and a commutative diagram
\begin{equation*}
 \label{eqn:partialVGIT}
\begin{tikzcd}
 S\arrow[d]\arrow[r] & S_K\arrow[r,"g"]\arrow[d] & \mathbb{A}^2/\Gamma\arrow[d] \\
\mathfrak{M}_\zeta(\Pi_\Gamma,\delta)\arrow[r] & \mathfrak{M}_{\zeta_K}(\Pi_\Gamma,\delta)\arrow[r] & \mathfrak{M}_0(\Pi_\Gamma,\delta)
 \end{tikzcd}
\end{equation*}
 where the vertical maps are isomorphisms, and the lower horizontal maps are given by VGIT.
\end{corollary}
\begin{proof}
Consider the Weyl chamber $C=\{\zeta\in \Theta_\delta\mid \zeta(\rho_i)>0\text{ for }i>0\}$ in $\Theta_\delta$, and let $\zeta\in C$. We have $\zeta_K\in \overline{C}$, so the VGIT morphisms on the bottom row are well-defined. Theorem~\ref{thm:Kronheimer} shows that $\mathfrak{M}_{\zeta_K}(\Pi_\Gamma,\delta)$ is the birational model of $S$ determined by the basepoint-free line bundle $L_C(\zeta_K)$. The degree of the line bundle $L_C(\zeta_K)$ on the $(-2)$-curve in $S$ indexed by node $i\in I$ is equal to $\zeta_K(\rho_i)$. Equation \eqref{eqn:zetaK} shows that this is zero if and only if $i\in K$, so $\mathfrak{M}_{\zeta_K}(\Pi_\Gamma,\delta)$ is isomorphic to the surface $S_K$ obtained from $S$ by contracting only the $(-2)$-curves indexed by $K$. 
  \end{proof}

\section{The Hilbert scheme as a quiver variety}

Let $n\geq 1$. The $n^\mathrm{th}$ symmetric product of $S_K$, denoted $\Sym^n(S_K)$, is the variety over $\kk$ obtained as the quotient of the $n$-fold product of $S_K$ by the action of the symmetric group on $n$ letters that permutes the factors. The \emph{Hilbert scheme of $n$-points in $S_K$}, denoted $\Hilb^n(S_K)$, is the fine moduli space of length $n$ subschemes in $S_K$. That is, $\Hilb^n(S_K)$ is the scheme over $\kk$ that represents the functor associating to each scheme $X$, the set of closed subschemes $Z\subset X\times S_K$ such that $Z$ is finite, flat and surjective of degree $n$ over $X$.  Throughout, we let $\Hilb^n(S_K)$ denote the underlying reduced subscheme of $\Hilb^n(S_K)$; in fact, the Hilbert scheme is reduced for $n\leq 7$ by \cite[Theorem~1.2]{CY26}.

The Hilbert-Chow morphism   
\[
  \pi\colon \Hilb^n(S_K)\longrightarrow \Sym^n(S_K)
  \] 
  associates to each length $n$ subscheme $Z$ in $S_K$, the effective 0-cycle $\sum_{x\in \supp(\mathcal{O}_Z)} \dim_{\kk}(\mathcal{O}_{Z,x}) x$, where $\supp(\mathcal{O}_Z)$ denotes the support of the sheaf $\mathcal{O}_Z$.

\subsection{The symmetric product as a quiver variety}
Our next goal is to prove that $\Sym^n(S_K)$ is a quiver variety. As a first step, we identify the relevant GIT cone in the space $\Theta_v$ of stability conditions on $\Pi$.

\begin{lemma}
 \label{lem:C_K'}
 For any subset $K\subseteq I=\{1,\dots, r\}$, the cone
\[
\cone{K}' = \big\{\theta\in \Theta_v \mid \theta(\delta) = 0, \theta(\rho_k)=0\text{ for }k\in K, \theta(\rho_i) \geq 0 \text{ for }i\in I\setminus K\big\}
\]
is a GIT cone of dimension $r-\vert K\vert$ that is contained in the simplicial cone $F$ from \eqref{eqn:F}.
\end{lemma}
\begin{proof}
The closure of the GIT chamber $\{\theta\in \Theta_v \mid \theta(\rho_i)>(n-1)\theta(\delta)>0 \text{ for all }1\leq i\leq r\}$ is a top-dimensional GIT cone in $F$ \cite[Example~2.1(1)]{BC20}. The cone $\cone{K}'$ is the face of this GIT cone given by intersecting with the supporting hyperplanes $\delta^\perp$ and $\rho_k^\perp$ for all $k\in K$. The face of each cone in a fan also lies in the fan, so $\cone{K}'$ is a GIT cone, and a dimension count gives $\dim\cone{K}' = (r+1)-1-\vert K\vert$.
\end{proof}
 
\begin{proposition}
\label{prop:SymtoSym}
Let $n\geq 1$. For any $\theta'_K\in \relint(\sigma_K')$, we have 
\[
\Sym^n(S_K) \cong \mathfrak{M}_{\theta'_K}(\Pi,v).
\]
 In particular, $\Sym^n(S_K)$ has symplectic singularities, its Picard rank is $r-\vert K\vert$, and variation of GIT quotient induces a symplectic partial resolution $\Sym^n(S_K)\rightarrow \Sym^n(\mathbb{A}^2/\Gamma)$.
\end{proposition}
\begin{proof}
Since GIT quotients for stability conditions in the relative interior of the same GIT cone are isomorphic, it suffices to show that $\Sym^n(S_K) \cong \mathfrak{M}_{\theta'_K}(\Pi,v)$ for some 
$\theta'_K\in \relint(\sigma_K^\prime)$. Consider the stability condition $\theta'_K\in \Theta_v$ for $\Pi$-modules determined by setting 
 \begin{equation}
 \label{eqn:theta}
 \theta'_K(\rho_i) = \left\{ \begin{array}{cr} 0 & \text{ for }i\in \{\infty\}\cup K \\ 1 & \text{for }i\in I\setminus K\end{array}\right..
 \end{equation} 
 Since $\theta'_K \in \Theta_v$, we have $n\theta'_K(\delta) = -\theta'_K(\rho_\infty)=0$, and hence $\theta'_K\in  \relint(\sigma_K')$. We may apply \cite[Proposition~4.2\two]{BC20} to see that the canonical decomposition of $v$ is $v=\rho_\infty+\delta+\cdots +\delta$, where $\delta$ appears $n$ times. The decomposition theorem of Bellamy--Schedler~\cite[Theorem~1.4]{BS21} gives
\begin{equation}
\label{eqn:decomp}
\mathfrak{M}_{\theta'_K}(\Pi,v)\cong \mathfrak{M}_{\theta'_K}(\Pi,\rho_\infty)\times \Sym^n\big(\mathfrak{M}_{\theta'_K}(\Pi,\delta)\big).
\end{equation}
There is a unique $\Pi$-module of dimension vector $\rho_\infty$ up to isomorphism, so the first factor here is a single closed point. 

We claim that $\mathfrak{M}_{\theta'_K}(\Pi,\delta)\cong S_K$. Indeed, when regarded as a vector in $\ZZ\oplus R(\Gamma)$, the component of $\delta$ at vertex $\infty$ equals zero, so the closed immersion $\iota\colon \Rep(\Pi_\Gamma,\delta)\hookrightarrow \Rep(\Pi, \delta)$ of representation schemes is actually a $G(\delta)$-equivariant isomorphism. Furthermore, the characters $
\zeta_K$ and $\theta'_K$ from \eqref{eqn:zetaK} and \eqref{eqn:theta} respectively are compatible, in the sense that $f\in \kk[\Rep(\Pi,\delta)]$ is $\theta'_K$-semi-invariant if and only $\iota^*(f)\in \kk[\Rep(\Pi_\Gamma,\delta)]$ is $\zeta_K$-semi-invariant. Therefore
\begin{equation}
    \label{eqn:zetaKthetaK}
 \mathfrak{M}_{\zeta_K}(\Pi_\Gamma,\delta)\cong \Rep(\Pi_\Gamma,\delta)\git_{\zeta_K} \! G(\delta) \cong \Rep(\Pi,\delta)\git_{\theta'_K} \! G(\delta) \cong \mathfrak{M}_{\theta'_K}(\Pi,\delta).
 \end{equation}
 The claim follows from Corollary~\ref{cor:S_K}, so $\mathfrak{M}_{\theta'_K}(\Pi,v)\cong \Sym^n(S_K)$. This proves the first statement. 
 
When $n=1$, the second statement follows from Corollary~\ref{cor:S_K} and the isomorphism \eqref{eqn:zetaKthetaK}. For $n>1$, the variety $\Sym^n(S_K)$ has symplectic singularities by \cite[(2.5)]{Beauville00}. Also, since $\theta'_K\in F$, Theorem~\ref{thm:BC20} gives that the VGIT morphism $\mathfrak{M}_{\theta_K'}(\Pi,v)\to \mathfrak{M}_{0}(\Pi,v)\cong \Sym^n(\mathbb{A}^2/\Gamma)$ is a symplectic partial resolution. 
 The calculation of the Picard rank follows from Corollary~\ref{cor:PicdimC} and Lemma~\ref{lem:C_K'}.
\end{proof}

\begin{corollary}
\label{cor:n=1}
  Theorem~\ref{thm:mainnew} holds when $n=1$.
\end{corollary}
\begin{proof}
 For $\theta_K\in \relint(\sigma_K)$ and $\theta_K^\prime\in \relint(\sigma_K^\prime)$, we claim there is a commutative diagram 
\begin{equation*}
 \label{eqn:Hilb1}
\begin{tikzcd}
 \Hilb^1(S_K)\arrow[d]\arrow[r,"\pi"] & \Sym^1(S_K)\arrow[d] \\
\mathfrak{M}_{\theta_K}(\Pi,v)\arrow[r] & \mathfrak{M}_{\theta_K^\prime}(\Pi,v)
 \end{tikzcd}
\end{equation*}
 in which all maps are isomorphisms. The Hilbert--Chow morphism at the top is an isomorphism, and the right-hand map is the isomorphism from Proposition~\ref{prop:SymtoSym}, so the first statement of Theorem~\ref{thm:mainnew} follows once we show that the VGIT morphism at the bottom of the diagram is an isomorphism. Since $n=1$, the hyperplane arrangement defining the GIT cone is $\mathcal{A}=\{\delta^\perp,\alpha^\perp \mid \alpha\in \Phi^+\}$. If $C\subseteq F$ denotes a GIT chamber containing $\sigma_K'$ in its closure, then we saw at the start of Section~\ref{sec:n=1} that $\ker(L_C)$ is spanned by $(-1,1,0,\dots, 0)\in \Theta_v$, so $L_C(\sigma_K)=L_C(\sigma_K^\prime)$. In particular, the birational models determined by $L_C(\theta_K)$ and $L_C(\theta_K')$ coincide, giving $\mathfrak{M}_{\theta_K}(\Pi,v)\cong  \mathfrak{M}_{\theta_K'}(\Pi,v)$, so the first statement of Theorem~\ref{thm:mainnew} holds. The geometric properties listed in the second statement of Theorem~\ref{thm:mainnew} all follow in the case $n=1$ from the isomorphism $\Hilb^1(S_K)\cong S_K$.
\end{proof}

\subsection{The Hilbert scheme is a quiver variety}
 From now on, assume $n>1$. Let $W \subseteq \Sym^n(S_K)$ denote the locus consisting of points for which the support of the corresponding cycle in $S_K$ contains at least one singular point of $S_K$. Notice that $\Sym^n(S_K) = \Sym^n(S_K^{\reg}) \sqcup W$, where $S_K^{\reg}$ denotes the non-singular locus in $S_K$.
  
 \begin{lemma}
 \label{lem:preimageW}
    The locus $W \subseteq \Sym^n(S_K)$ is of codimension two. Moreover, the preimage $\pi^{-1}(W)$ has codimension two in $\Hilb^n(S_K)$, and the complement of $\pi^{-1}(W)$ in $\Hilb^n(S_K)$ is $\Hilb^n(S_K^{\reg})$. 
  \end{lemma}
  \begin{proof}
  Since $S_K$ has symplectic singularities, this follows from \cite[Lemma~6.9]{CY26}.
  \end{proof}

\begin{theorem}\label{thm:HilbisNQV}
  There is a GIT cone $\tau_K$ satisfying $\sigma'_K\subsetneq \tau_K\subseteq F$, that is not contained in $\delta^\perp$, such that $  \Hilb^n(S_K)\cong \mathfrak{M}_{\theta_K}(\Pi,v)$ for $\theta_K\in \relint(\tau_K)$. In particular, variation of GIT quotient for quiver varieties gives a crepant partial resolution   $\Hilb^n(S_K)\to \Sym^n(\mathbb{A}^2/\Gamma)$.
  \end{theorem}
\begin{proof}
  Since $S_K$ has canonical singularities, it follows from Craw--Yamagishi~\cite[Theorem~1.1]{CY26} that $\Hilb^n(S_K)$ has canonical singularities. Kaplan--Schedler~\cite[Section~4.3]{KS24} leverages this result to construct a unique (globally) projective crepant resolution of $\Hilb^n(S_K)$ such that the composition with the Hilbert--Chow morphism $\pi\colon \Hilb^n(S_K)\to \Sym^n(\mathbb{A}^2/\Gamma)$ is a crepant resolution, so $\pi$ itself is a crepant partial resolution.  It follows from Theorem~\ref{thm:BC20} that $\Hilb^n(S_K)\cong \mathfrak{M}_{\theta_K}(\Pi,v)$ for some $\theta_K\in F$. Let $\tau_K\subseteq F$ denote the unique GIT cone that contains $\theta_K$ in its relative interior. The Hilbert--Chow morphism fits into a commutative diagram
      \begin{equation}
 \label{eqn:crepantVGIT}
\begin{tikzcd}
  \Hilb^n(S_K) \ar[d]\ar[r,"\pi"] &  \Sym^n(S_K)\ar[d] \\
  \mathfrak{M}_{\theta_K}(\Pi,v) \ar[r] & \mathfrak{M}_{\theta'_K}(\Pi,v)
  \end{tikzcd}
  \end{equation}
  in which both vertical maps are isomorphisms, and the lower morphism is given by variation of GIT quotient. As a result, the GIT cone $\cone{K}'$ that determines $\mathfrak{M}_{\theta'_K}(\Pi,v)
$ is a face of the GIT cone $\tau_K$ that determines $ \mathfrak{M}_{\theta_K}(\Pi,v) $. Combining diagram \eqref{eqn:crepantVGIT} with a similar diagram for the morphism $\Sym^n(S_K)\rightarrow \Sym^n(\mathbb{A}^2/\Gamma)$ from Proposition~\ref{prop:SymtoSym} shows that the desired crepant partial resolution is obtained by variation of GIT quotient for quiver varieties.

It remains to show $\tau_K$ is not contained in the hyperplane $\delta^\perp$. Proposition~\ref{prop:SymtoSym} shows that each face of the cone  $F\cap \delta^\perp$ is of the form $\sigma_L'$ for some $L\subseteq \{1,\dots, r\}$, and 
 defines the quiver variety $\mathfrak{M}_{\theta_L}(\Pi,v)\cong \Sym^n(S_L)$. Suppose for a contradiction that $\tau_K\subset \delta^\perp$ with $\sigma_K^\prime\subsetneq \tau_K$. Then $\tau_K=\sigma_L'$ for some $L \subsetneq K$ satisfying $\Hilb^n(S_K)\cong \Sym^n(S_L)$.
 Variation of GIT quotient for quiver varieties shows there is a commutative diagram 
     \begin{equation}
     \label{eqn:HilbnotSym}
\begin{tikzcd}
  \Hilb^n(S_K) \ar[dr,"\pi"]\ar[rr,"\varepsilon"] & & \Sym^n(S_L)\ar[dl,swap,"\pi'"] \\
  & \Sym^n(S_K) &
  \end{tikzcd}
  \end{equation}
in which $\varepsilon$ is an isomorphism  and $\pi'$ is obtained by applying the functor  $\Sym^n(-)$ to the crepant partial resolution $S_L\to S_K$. There exists $k\in K\setminus L$ that determines a rational curve $\ell_k$ contracted by the morphism $S_L\to S_K$. Let $W_k\subseteq \Sym^n(S_L)$ denote the locus parametrising cycles whose support contains a point in the rational curve $\ell_k$. Then $W_k$ contains the locus parametrising $n$ distinct points in the surface $S_L$ such that one of the points must lie in the curve $\ell_k$. This locus is of codimension-one in $\Sym^n(S_L)$, so $W_k$ and hence $(\pi')^{-1}(\pi'(W_k))$ contains a divisor. However, $\pi'(W_k)$ is contained in the locus $W\subseteq \Sym^n(S_K)$ from Lemma~\ref{lem:preimageW}. Since $\pi^{-1}(W)$ has codimension two,  it follows that $\pi^{-1}(\pi'(W_k))$ has codimension at least two. However, this contradicts the fact that $\varepsilon$ identifies $(\pi')^{-1}(\pi'(W_k))$ and $\pi^{-1}(\pi'(W_k))$, so $\tau_K$ is not contained in $\delta^\perp$ after all.
\end{proof}

\subsection{The dimension of the GIT cone}
Our next goal is to compute the dimension of $\tau_K$ when $n>1$. Corollary~\ref{cor:PicdimC} shows that this coincides with the Picard rank of $\Hilb^n(S_K)$. Initially, we consider the class group over the rationals:
 
 \begin{lemma}
 \label{lem:ClPicS_K}
 Let $n>1$. There is an isomorphism $\Class(\Hilb^n(S_{K}))_{\QQ} \cong \Pic\big(\Sym^n(S_{K})\big)_{\QQ}\oplus \QQ$.
 \end{lemma}
 \begin{proof} 
Recall that the complement of $\Sym^n(S_K^{\reg})$ in $\Sym^n(S_K)$ is the locus $W$ that has codimension two by Lemma~\ref{lem:preimageW}. The same lemma shows that $\pi^{-1}(W)$ has codimension two in $\Hilb^n(S_K)$ and, moreover, that the complement of $\pi^{-1}(W)$ in $\Hilb^n(S_K)$ is $\Hilb^n(S_K^{\reg})$.  Since $\Hilb^n(S_K)$ is normal \cite{CY26} and $\Hilb^n(S_{K}^{\reg})$ is non-singular~\cite{Fogarty68}, there is a $\ZZ$-linear isomorphism
\begin{equation}
    \label{eqn:Clregular}
\Class\!\big(\!\Hilb^n(S_{K})\big) \cong \Pic\!\big(\!\Hilb^n(S_{K}^{\reg})\big). 
\end{equation}
Since $n>1$, Iarrobino~\cite[Corollary~3]{Iarrobino72} shows that 
\begin{equation}
    \label{eqn:Iarrobino}
\Pic\!\big(\!\Hilb^n(S_{K}^{\reg})\big)_{\QQ}\cong \Pic\!\big(\!\Sym^n(S_{K}^{\reg})\big)_{\QQ} \oplus \QQ.
\end{equation}
The surfaces $S_K^{\reg}$ and $S_K$ are non-singular and have Kleinian singularities respectively. Therefore the varieties $\Sym^n(S_K^{\reg})$ and $\Sym^n(S_K)$ both have quotient singularities, so they are normal and $\QQ$-factorial \cite[Proposition~5.15]{KM98}. It follows that
\begin{equation}
\label{eqn:PicClPic}
\Pic\!\big(\!\Sym^n(S_{K}^{\reg})\big)_{\QQ} \cong \Class\!\big(\!\Sym^n(S_{K}^{\reg})\big)_{\QQ} \cong \Class\!\big(\!\Sym^n(S_{K})\big)_{\QQ} \cong\Pic\!\big(\!\Sym^n(S_{K})\big)_{\QQ},
\end{equation}
 where we deduce the second isomorphism by (restriction and) uniqueness of extension of reflexive sheaves across a codimension-two locus in a normal variety~\cite[Proposition~1.6]{Hartshorne80}. Now combine \eqref{eqn:Clregular}, \eqref{eqn:Iarrobino} and \eqref{eqn:PicClPic} to obtain the desired isomorphism of $\QQ$-vector spaces. 
 \end{proof}

 \begin{proposition}
 \label{prop:PicHilb}
  Let $n>1$. The Picard group of $\Hilb^n(S_K)$ has rank $r-\vert K\vert + 1$. 
  In particular, the GIT cone $\tau_K$ from Theorem~\ref{thm:HilbisNQV} has dimension $r-\vert K\vert + 1$.  
 \end{proposition}
 \begin{proof}
 Since $\Hilb^n(S_K)$ is normal, the Picard group of  $\Hilb^n(S_K)$ is a subgroup of the Class group, so the Picard rank of $\Hilb^n(S_K)$ is at most $r-\vert K\vert + 1$ by Proposition~\ref{prop:SymtoSym} and Lemma~\ref{lem:ClPicS_K}.  On the other hand, since $\tau_K$ contains $\sigma_K'$, Lemma~\ref{lem:C_K'} gives $\dim (\tau_K) \geq r-\vert K\vert$. For part \one, it remains to show that $\dim(\tau_K)\neq r-\vert K\vert$. Suppose otherwise. Then Theorem \ref{thm:HilbisNQV} implies that $\tau_K=\sigma_K'$, but this gives $\Hilb^n(S_K) \cong \Sym^n(S_K)$. We showed following \eqref{eqn:HilbnotSym} that this is absurd. 
\end{proof}
 
\begin{corollary}
\label{cor:Qfactorial}
Let $n\geq 1$. The variety $\Hilb^n(S_K)$ is $\QQ$-factorial. 
\end{corollary}
\begin{proof}
 For $n>1$, combine Propositions~\ref{prop:SymtoSym} and \ref{prop:PicHilb}  with Lemma~\ref{lem:ClPicS_K} to see that the ranks of the Picard group and Class group of $\Hilb^n(S_K)$ agree. For $n=1$, see the proof of Corollary~\ref{cor:n=1}.
\end{proof}

 At this stage,  we can already prove Theorem~\ref{thm:mainnew} in the special case $K= \varnothing$, thereby providing an independent proof of the main result of Kuznetsov~\cite{Kuznetsov07} (also established separately by both Haiman and Nakajima). In this case, the cone $\sigma_K$ from Theorem~\ref{thm:mainnew} is
    \[
    \sigma_\varnothing\coloneqq \big\{\theta\in \Theta_v \mid   \theta(\rho_i) \geq (n-1) \theta(\delta) \text{ for }1\leq i\leq r, \; \theta(\delta) \geq 0\big\}. 
    \]

\begin{corollary}[Kuznetsov] \label{cor:Kuznetsov}
Let $n\geq 1$ and let $S\rightarrow \mathbb{A}^2/\Gamma$ denote the minimal resolution. There is an isomorphism $\Hilb^n(S)\cong \mathfrak{M}_\theta(\Pi,v)$ for any $\theta$ in the interior of $\sigma_\varnothing$.
\end{corollary}
\begin{proof}
We may assume $n>1$, so we need only show that $\sigma_\varnothing = \tau_\varnothing$ by  Theorem~\ref{thm:HilbisNQV}. The dimension of $\tau_\varnothing$ equals $r+1$ by Proposition~\ref{prop:PicHilb}, and $\tau_\varnothing$ contains the GIT cone
 \[
\cone{\varnothing}' = \big\{\theta\in \Theta_v \mid \theta(\delta) = 0, \theta(\rho_i) \geq 0 \text{ for }1\leq i\leq r\big\},
\] 
again by Theorem~\ref{thm:HilbisNQV}. We claim there is a unique GIT cone in $F$ with these properties. Indeed, the cone $\sigma_\varnothing^\prime$ has dimension $r$ and is equal to the codimension-one face $F\cap \delta^\perp$ of $F$, so there can be only one GIT cone of dimension $r+1$ in $F$ containing $\sigma_\varnothing^\prime$. It remains to notice that $\sigma_\varnothing$ is a GIT cone of dimension $r+1$ as it is the closure of the GIT chamber from \cite[Example 2.1(1)]{BC20}, and it contains $\sigma_\varnothing^\prime$, so $\sigma_\varnothing = \tau_\varnothing$ as required.
\end{proof}

\subsection{Combinatorial characterisation of the GIT cone}  \label{sec:combinatorial}
  To complete the proof of the first statement in Theorem~\ref{thm:mainnew} for any subset $K\subseteq \{1,\dots,r\}$, we need only show for $n>1$ that the cone $\tau_K$ introduced in Theorem~\ref{thm:HilbisNQV} coincides with the cone $\sigma_K$ from Theorem~\ref{thm:mainnew}.
 
 Our combinatorial approach to this question begins by associating a top-dimensional GIT cone in $F$ to each subset $K\subseteq \{1,\dots, r\}$. 

\begin{lemma}
\label{lem:CK*} 
  For any $K\subseteq I=\{1,\dots,r\}$, the open cone
  \[ 
  C_K \coloneqq \left\{ \theta \in \Theta_v \mid 
  \theta(\rho_k) > 0 \text{ for } k\in K, \;
  \theta(\rho_i) > (n-1)\theta(\delta) \text{ for } i\in I\setminus K, \; 
  \theta(\delta) > \sum_{k\in K} \delta_k \theta(\rho_k)
  \right\}
  \]
  is a GIT chamber. In particular, its closure $\overline{C_K}$ is a GIT cone in $F$ of dimension $r+1$.
 \end{lemma}
 \begin{proof}
 The result holds for $K = \varnothing$ because $\overline{C_K} =\sigma_K$ in that case. For $K\neq \varnothing$, let  $h:=\sum_{0\leq i\leq r} \delta_i$ denote the Coxeter number of $\Phi$. Define $\theta\in F$ by setting $\theta(\rho_k)=1$ for $k\in K$, $\theta(\delta)=h$ and $\theta(\rho_i) = (n-1)h+1$ for $i\in I\setminus K$. Then $\theta\in C_K$, so $C_K$ is non-empty. The supporting hyperplanes of $\overline{C_K}$ all lie in the arrangement $\mathcal{A}$, so it suffices to show no hyperplane of $\mathcal{A}$ intersects $C_K$. 
  
  For this, let $\theta\in C_K$ be arbitrary. Then  $\theta(\delta)> 0$ and $\theta(\rho_i)>0$ for $1\leq i\leq r$, so $C_K\subset F$ and $\theta$ does not lie in $\delta^\perp$. As noted following \eqref{eqn:F}, 
  every other hyperplane in $\mathcal{A}$ that intersects the interior of $F$ is of the form $(m\delta-\alpha)^\perp$ for some $\alpha\in \Phi^+$ and $0\leq m<n$. Let $\Phi^+_K$ denote the set of positive roots $\alpha \in \Phi^+$ that are of the form $\alpha=\sum_{k\in K} a_k \rho_k$ for some $0\leq a_k\leq \delta_k$. Consider three cases:  if $\alpha\in \Phi^+_K$ and $0 < m < n$, then $m\theta(\delta)> \sum_{k\in K} \delta_k \theta(\rho_k) \geq \theta(\alpha)>0$, so $\theta$ does not lie in $(m\delta-\alpha)^\perp$; if $\alpha\in \Phi^+_K$ and $m=0$, then $\theta(\alpha) \geq \theta(\rho_k)>0$ for some $k \in K$, so $\theta$ does not lie in $\alpha^\perp$; and if $\alpha \in \Phi^+ \setminus \Phi^+_K$ and $0 \leq m < n$, then $m\theta(\delta) \leq (n-1)\theta(\delta) < \theta(\rho_i) \leq \theta(\alpha)$ for some $i \in I \setminus K$, so $\theta$ does not lie in $(m\delta-\alpha)^\perp$. Therefore, no hyperplane in $\mathcal{A}$ intersects $C_K$ as required. 
 \end{proof}

 We now show that the cone $\sigma_K$ from Theorem~\ref{thm:mainnew} shares the same properties as the cone $\tau_K$ from Theorem~\ref{thm:HilbisNQV} (compare Theorem~\ref{thm:HilbisNQV}, Proposition~\ref{prop:PicHilb} and Corollary~\ref{cor:Qfactorial}).
 
\begin{proposition}
\label{prop:C_Kproperties}
 For any subset $K\subseteq I\coloneqq\{1, \dots, r\}$, the cone 
    \[
    \sigma_K\coloneqq \big\{\theta\in \Theta_v \mid  \theta(\rho_k)=0\text{ for }k\in K,\; \theta(\rho_i) \geq (n-1) \theta(\delta) \text{ for }i\in I\setminus K, \; \theta(\delta) \geq 0\big\}
    \]
    is a face of the GIT cone $\overline{C_K}$. Moreover, $\sigma_K$ satisfies the following properties:
    \begin{enumerate}
        \item[\one] it is a GIT cone of dimension $r-\vert K \vert + 1;$
        \item[\two] it contains $\sigma_K^\prime$ and it is not contained in the hyperplane $\delta^\perp;$
        \item[\three] for $\theta_K\in \relint(\sigma_K)$, the Nakajima quiver variety $\M_{\theta_K}(\Pi, v)$ is $\QQ$-factorial.
    \end{enumerate}
\end{proposition}
\begin{proof}
    Note that $\sigma_K$ is the face of the GIT cone $\overline{C_K}$ from Lemma~\ref{lem:CK*} obtained by intersecting with the supporting hyperplanes $\rho_k^\perp$ of $\overline{C_K}$ for all $k\in K$. Part \one\ is therefore clear, while \two\ follows from the equality $\sigma_K^\prime = \sigma_K\cap \delta^\perp$.
    For part \three, we notice that the cone $\sigma_K$ lies in the face $F'=\{ \theta \in \Theta_v \mid \theta(\rho_k)=0, k \in K\} $ of $F$, which is $r- \vert K \vert +1$ dimensional. Part \one\ shows that $\sigma_K$ is also of dimension $r- \vert K \vert +1$, so $\M_{\theta_K}(\Pi, v)$ is $\QQ$-factorial by Corollary~\ref{cor:PicdimC}\two.
\end{proof}

 \begin{proof}[Proof of Theorem~\ref{thm:mainnew}]
 We may assume that $n>1$. To prove the first statement, it remains to show that the cone $\tau_K$ from Theorem~\ref{thm:HilbisNQV} coincides with $\sigma_K$. Combining Theorem~\ref{thm:HilbisNQV}, Proposition~\ref{prop:PicHilb} and Corollary~\ref{cor:Qfactorial} shows that $\tau_K$ satisfies properties \one-\three\ from Proposition~\ref{prop:C_Kproperties}. Since $\sigma_K$ also satisfies these properties, it remains to show that only one GIT cone satisfies these three properties.

 For this, let $\sigma$ be any GIT cone satisfying properties \one-\three\ from Proposition~\ref{prop:C_Kproperties}, and let $F'\subseteq F$ be the face of minimal dimension satisfying $\sigma \subseteq F'$. Let $\theta\in \relint(\sigma)$, so also $\theta\in \relint(F')$. Since $\M_\theta(\Pi,v)$ is $\QQ$-factorial, Corollary~\ref{cor:PicdimC} gives that
 \begin{equation}
 \label{eqn:dimFdimC}
 \dim(F')=\dim(\sigma) = r-\vert K\vert +1.
 \end{equation}
 Lemma~\ref{lem:C_K'} shows that $\dim \sigma_K^\prime = r-\vert K\vert$, $\sigma_K^\prime \subseteq \sigma$ by assumption and $\sigma_K^\prime$ is a face of $F$, so $\sigma_K^\prime$ is a codimension-one face of $F^\prime$. Now, $F$ is a simplicial cone of dimension $r +1$, so $\vert K \vert +1$ faces of dimension $r- \vert K \vert +1$ in $F$ contain $\sigma_K'$. Precisely $\vert K \vert$ of these faces lie in $F \cap \delta^\perp$ because $\sigma_K' \subseteq \delta^\perp$, so $F^\prime$ is the unique $r- \vert K \vert +1$-dimensional face of $F$ that contains $\sigma_K^\prime$ but does not lie in $\delta^\perp$. Thus
 \[
 F^\prime =  \big\{\theta\in F \mid  \theta(\rho_k)=0\text{ for }k\in K, \; \theta(\delta) \geq 0\big\}.
 \]
 By definition, $\sigma$ is a top-dimensional cone in $F'$, and $\sigma_K' = F' \cap \delta^\perp$ is contained in a unique supporting hyperplane of $F'$, so $\sigma$ is uniquely determined. Therefore $\tau_K=\sigma_K$.

 For the second statement, $S_K$ has symplectic singularities, so $\Hilb^n(S_K)$ is a normal variety with symplectic singularities by \cite{CY26} (this also follows directly via the isomorphism $\Hilb^n(S_K)\cong \mathfrak{M}_{\theta_K}(\Pi,v)$ by results from Bellamy--Schedler~\cite{BS21}). Corollary~\ref{cor:Qfactorial} gives that $\Hilb^n(S_K)$ is $\QQ$-factorial, so it remains to show that $\mathfrak{M}_{\theta_K}(\Pi,v)$ admits a unique projective, symplectic resolution. For this, it suffices by \cite[Theorem~1.2]{BC20} to show that there is a unique GIT chamber in $F$ containing $\sigma_K$ in its closure. Recall from Proposition~\ref{prop:C_Kproperties} that $C_K$ is one such chamber. For uniqueness, let $C$ be another such chamber.  If there exists $\alpha\in \Phi^+_K$ and $0 \leq m < n$ such that $\theta(m\delta\pm\alpha)<0$, then $m\theta(\delta)<0$, contradicting $C\subset F$. Alternatively, if there exists $\alpha \in \Phi^+ \setminus \Phi^+_K$ and $0 \leq m < n$ such that $\theta(m\delta-\alpha) >0$, then $(n-1)\theta(\delta) \geq m\theta(\delta)>\theta(\alpha)\geq \theta(\rho_j)$ for some $j\in J$. Since $\theta$ is arbitrary in $C$, any $\theta_K\in \relint(\sigma_K)$ must satisfy the resulting weak inequality $(n-1)\theta_K(\delta)\geq \theta_K(\rho_j)$, a contradiction. Since we have analysed all the hyperplanes cutting the interior of $F$, the chamber $C$ satisfies precisely the same defining inequalities as $C_K$ after all.
 \end{proof}

 \section{Computing the nef and movable cones}
  We now use our description of $\Hilb^n(S_K)$ as the Nakajima quiver variety $\mathfrak{M}_{\theta_K}(\Pi,v)$ defined by $\theta_K\in \relint(\sigma_K)$ to compute the nef and movable cones of $\Hilb^n(S_K)$. For the chamber $C_K$ from Lemma~\ref{lem:CK*} that contains $\sigma_K$ in its closure, and for $\theta\in C_K$, write
 \begin{equation}
 \label{eqn:hn}
h\colon \mathfrak{M}_{\theta}(\Pi,v)\longrightarrow \mathfrak{M}_{\theta_K}(\Pi,v)
 \end{equation}
 for the morphism obtained by variation of GIT quotient. For $i\in Q_0$, write $\mathcal{R}_i$ for the corresponding tautological bundle on $\mathfrak{M}_{\theta}(\Pi,v)$, and write $\delta_J:=(\dim \rho_j)\in \NN^J$ where $J\coloneqq \{0,1,\dots, r\}\setminus K$.

 \begin{proof}[Proof of Theorem~\ref{thm:NefHilb}]
 The nef cone is of top-dimension in $\Pic(\Hilb^n(S_K))\otimes_\ZZ \QQ$ because $\Hilb^n(S_K)$ is projective over $\Sym^n(\mathbb{A}^2/\Gamma)$, so part \one\ follows once we prove part \two. For this, the proof of Corollary~\ref{cor:PicdimC}\one\ shows that pulling back along a VGIT morphism of the form $h$ defines an injective map on rational Picard groups that identifies $\Nef(\mathfrak{M}_{\theta_K}(\Pi,v))$ with the cone 
 \[
 L_{C_K}(\sigma_K) = \Big\{\bigotimes_{0\leq i\leq r}\det(\mathcal{R}_i)^{\otimes \vartheta_i} \mid \vartheta\in \sigma_K\Big\}.
 \] 
Part \two\ follows from the definition of $\sigma_K$ and the isomorphism  $\Hilb^n(S_K)\cong \mathfrak{M}_{\theta_K}(\Pi,v)$ from Theorem~\ref{thm:mainnew}. For part \three, let $F'$ denote the minimal face of $F$ containing $\sigma_K$. Since $\dim \sigma_K = \dim F'$, equation \eqref{eqn:pullbackMov} in the proof of Corollary \ref{cor:PicdimC}\two\ shows that the pullback of the movable cone of $\mathfrak{M}_{\theta_K}(\Pi,v)$ is identified with
 \[ L_{C_K}(F') = \Big\{\bigotimes_{0\leq i\leq r}\det(\mathcal{R}_i)^{\otimes \vartheta_i} \mid \vartheta\in F'\Big\}. \] 
 The definition of the face $F'$ and the isomorphism $\Hilb^n(S_K)\cong \mathfrak{M}_{\theta_K}(\Pi,v)$ from Theorem~\ref{thm:mainnew} shows that the movable cone of $\Hilb^n(S_K)$ is identified with
  \[
         \Big\{\textstyle{\bigotimes_{j\in J} \det(\mathcal{R}_j)^{\otimes \eta_j}} \mid 
     \eta(\rho_j) \geq 0 
     \text{ for }j\neq 0, \; 
     \eta(\delta_J)\geq 0\Big\}
    \]
as required.
For the fan decomposition, the face $F'$ is decomposed by the hyperplanes in $\mathcal{A}$ that intersect $\relint(F')$; these are precisely   
 $(m\delta_J- \alpha)^\perp$ for $\alpha\in \Phi_{I\setminus K}^+$ and $0\leq m < n$, where $\Phi_{I\setminus K}^+$ is the set of positive roots supported in $I\setminus K=J\setminus \{0\}$.
 \end{proof}

\small{
\bibliographystyle{plain}

}
\end{document}